\documentclass{birkjour}
 \usepackage{mathabx}
 \usepackage{hyperref}
 \usepackage[dvipsnames]{xcolor}
 \usepackage{xcolor}
 \usepackage{mathrsfs}
 \newcommand{\nc}{\mathbb{N}}

\usepackage{hhline}
\newtheorem{theorem}{Theorem}[section]
\newtheorem{defin}[theorem]{Definition}
\newtheorem{lemma}[theorem]{Lemma}

\newtheorem{observation}[theorem]{Observation}

\newtheorem{rem}{Remark}[section]
\newtheorem{pro}{Proposition}[section]

\newenvironment{defn*}{\begin{definition}}{\end{definition}}

 \numberwithin{equation}{section}

\begin{document}

%
%
%
%
%
%
%
%
%

\title[Tame Factorization Property I]
 {Tame Factorization Property I}


\author{Buket Can Bahadır}
\email{canbuket@gmail.com}
\author{Nazl\i\;Do\u{g}an}
\address{Fatih Sultan Mehmet Vakif University, 
    34445 Istanbul, Turkey}
\email{ndogan@fsm.edu.tr}



\dedicatory{Dedicated to the memory of Tosun Terzio\u{g}lu \\ on the 10th anniversary of his passing.}

\begin{abstract}
We introduce the \emph{tame factorization property} for triples of Fr\'echet spaces: a triple $(E,G,F)$ has the tame factorization  property, denoted by $(E,G,F) \in \mathfrak{TF}$, if there exists a nondecreasing function $S:\mathbb{N}\to\mathbb{N}$ such that every operator from $E$ to $F$ factoring through $G$ is $S$-tame. We give a complete characterization of $\mathfrak{TF}$ in terms of operator seminorm estimates. We specialize this characterization to triples in which one or more of the spaces are K\"othe spaces, putting explicit conditions on the K\"othe matrices, and we describe how $\mathfrak{TF}$ behaves under projective tensor products of K\"othe spaces. Finally, we show that a triple of K\"othe spaces has the tame factorization property if and only if the family of operators that factor as products of two quasi-diagonal operators is tame.
\end{abstract}

\maketitle

\section{Introduction}

Tameness provides a way to control and predict the behavior of operators defined between Fr\'echet spaces: operators between tame space pairs are subject to a uniform estimate, instead of growing arbitrarily. This regularity is useful both when working directly with linear operators and when studying the structure of Fr\'echet spaces.

 The notion goes back to Hamilton's 1982 work on linearly tame spaces in the context of the Nash--Moser inverse function theorem \cite{hamil}, and was given its general definition for pairs of Fr\'echet spaces by Dubinsky and Vogt \cite{vogt1}, who characterized tame power series spaces of infinite type. Nyberg \cite{nyberg} classified tame pairs of power series spaces, and later Piszczek \cite{pisek2} characterized tame pairs $(X,Y)$ in which one of the spaces is a power series space, in terms of topological invariants.  Aytuna \cite{A} showed that a nuclear Fr\'echet space $E$ with properties $\underline{DN}$ and $\Omega$, and a stable finitely nuclear associated exponent sequence, is isomorphic to a power series space of finite type if and only if $E$ is tame and the approximate diametral dimensions of $E$ and the corresponding power series space coincide, and used this result to characterize the tameness of spaces of analytic functions on Stein manifolds. Recently, Bahad\i r \cite{can} gave a complete classification of tame pairs of power series spaces and also showed that tameness of a pair of K\"othe spaces $(\lambda(A),\lambda(B))$ is determined entirely by the tameness of the quasi-diagonal operators between them.

 The study of tameness was preceded by a considerably older question in the theory of locally convex spaces: that of boundedness, namely the phenomenon whereby every continuous linear operator between two locally convex spaces is necessarily bounded. This phenomenon has been studied extensively, some relevant references are \cite{A1, B1, B2, djakov1, djakov2, OT1, nurlu, nurlu2, vogt2}. Vogt \cite{vogt2} gave a complete characterization of the pairs $(E,F)$ of Fr\'echet spaces, denoted by $(E,F)\in \mathfrak{B}$, for which every continuous linear operator from $E$ to $F$ is bounded. 
 
 The boundedness relation was subsequently extended to a factorized setting. Zahariuta's isomorphism results for Cartesian products of K\"othe spaces \cite{Z1,Z2} led Djakov, Terzio\u{g}lu, Yurdakul, and Zahariuta \cite{djakov1} to isolate a factorized counterpart of $\mathfrak{B}$, the \emph{bounded factorization property}: a triple $(E,G,F)$ has property bounded factorization property, denoted by $(E,G,F)\in \mathfrak{BF}$, if every operator from $E$ to $F$ that factors through $G$ is bounded. This relation was used to obtain new isomorphic classifications of Cartesian products of K\"othe and power series spaces, and, for tensor products of K\"othe spaces, was shown to follow from boundedness of the corresponding pairs \cite{djakov1}. Terzio\u{g}lu and Zahariuta \cite{terzi2} subsequently gave a complete characterization of bounded factorization property for arbitrary Fr\'echet spaces, via a bilinear version of Grothendieck's factorization theorem. Terzio\u{g}lu, Yurdakul, and Zahariuta \cite{terzi1} also showed that a triple of K\"othe spaces fails to have bounded factorization property if and only if there is an unbounded operator factoring, as a product of two quasi-diagonal operators, through the middle space.

The present paper combines these two ideas by introducing the \emph{tame factorization property}. We define the tame factorization property as: a triple $(E,G,F)$ has the tame factorization property, denoted by $(E,G,F)\in \mathfrak{TF}$, if every operator from $E$ to $F$ factoring through $G$ is tame.

Our main result, Theorem~4.2, characterizes the tame factorization property for arbitrary Fr\'echet spaces via the same bilinear factorization technique used by Terzio\u{g}lu and Zahariuta for the bounded factorization property in \cite{terzi2}. When $E=G$ or $G=F$, this recovers Piszczek's characterization of tame pairs \cite{pisek2}. We then specialize this characterization to the case where one or more of $E,G,F$ are K\"othe spaces, Propositions~4.2--4.5, Theorem~4.4, giving explicit matrix-theoretic conditions, and describe the behavior of the tame factorization property under projective tensor products of K\"othe spaces, Theorems~4.5--4.6. Finally,  we show that a triple of K\"othe spaces has the tame factorization property if and only if the family of operators that factor as products of two quasi-diagonal operators is tame  by Theorem~5.3.

We also mention that a companion paper, \emph{Tame Factorization Property II} \cite{CD}, is devoted to a further aspect of this theory. Vogt \cite{vogt2} showed that, when one of the spaces $E$ or $F$ is a power series space, the relation $(E,F) \in \mathfrak{B}$ can be characterized in terms of DN-$\Omega$ type topological invariants of the other space. Piszczek \cite{pisek2} later established an analogous characterization for tame pairs: when one of $E$ or $F$ is a power series space, the tameness of $(E,F)$ can likewise be expressed in terms of DN-$\Omega$ type topological invariants of the other space.  Motivated by these results, in \cite{CD} we consider triples $(E,G,F) \in \mathfrak{TF}$ in which one of the spaces is a power series space and another is a nuclear Fr\'echet space with properties $\underline{DN}$ and $\Omega$, and under a suitable diametral or approximate diametral dimension condition, we investigate the conditions that the third space must satisfy for the triple to have the tame factorization property under some diametral/approximate diametral dimension condition.

\section{Preliminaries}
\label{chp:preliminaries}

In this section, we introduce the basic definitions, terminology, and notation.  We refer to \cite{vogt4} for complementary background, additional details, and further discussion of related concepts.
\subsection{Fréchet Spaces}

A complete Hausdorff locally convex space E whose topology defined by countable fundamental
system of seminorms $(\|\cdot\|_{k})_{k\in \mathbb{N}}$ is called a \textit{Fr\'echet space}. A \textit{grading} on a Fr\'echet space $E$ is an increasing sequence of seminorms $(\|\cdot\|_{k})_{k\in \mathbb{N}}$ that defines the topology of E. Every Fr\'echet space admits a grading. A Fréchet space together with a chosen grading is called a \textit{graded Fréchet space}.
We refer the reader to \cite{hamil} for further details.
In the sequel, we work exclusively with graded Fr\'echet spaces unless stated otherwise.

Let $E$ be a Fréchet space with defining seminorms $(\| \cdot
\|_k )_{k\in \nc}$. We denote by $E'$ its topological dual, endowed with the dual seminorms ${\|\cdot\|_k^*}$ given by
\[\hspace{1.45in}\|u\|^*_k=\sup_{\|x\|_k\leq 1}|u(x)|,\hspace{1.45in} u\in E'.\]

A matrix $(a_{n,k})_{k,n\in \mathbb{N}}$ of non-negative scalars is called a \textit{Köthe matrix} if $ a_{n,k}\leq a_{n,k+1}$ for all $n,k\in \mathbb{N}$, and 
for every $n\in \mathbb{N}$ there exists a $k\in \mathbb{N}$ such that $a_{n,k}>0$. The \textit{Köthe spaces} considered in this paper are defined as follows, where $\mathbb{K}$ denotes either $\mathbb{R}$ or $\mathbb{C}$:
\begin{align*}
\lambda(A)
&=\Bigl\{ x=(x_{j})_{j\in \nc}\in \mathbb{K}^{\mathbb{N}} :
\|x\|_{k}
=\sum_{j=1}^{\infty} |x_{j}| a_{j,k}
<\infty,\ \forall k\in\mathbb{N} \Bigr\}, \\[0.3em]
\lambda^{\infty}(A)
&=\Bigl\{ x=(x_{j})_{j\in \nc}\in \mathbb{K}^{\mathbb{N}} :
\|x\|_{k}
=\sup_{j\in\mathbb{N}} |x_{j}| a_{j,k}
<\infty,\ \forall k\in\mathbb{N} \Bigr\}, \\[0.3em]
\lambda^{\infty}_{0}(A)
&=\Bigl\{ x\in \lambda^{\infty}(A) :
\lim_{j\to\infty} x_{j} a_{j,k}=0,\ \forall k\in\mathbb{N} \Bigr\}. \\[0.3em]
\end{align*}
Let $\alpha=\left(\alpha_{n}\right)_{n\in \mathbb{N}}$ be a non-negative increasing sequence with $\displaystyle \lim_{n\rightarrow \infty} \alpha_{n}=+\infty$.
\textit{The power series spaces of infinite and finite type} are defined as follows: 
\begin{align*}
\Lambda_{\infty}(\alpha)
&=\Bigl\{ x=(x_{j})_{j\in \nc}\in \mathbb{K}^{\mathbb{N}} :
\|x\|_{k}
=\sum_{j=1}^{\infty} |x_{j}|e^{k\alpha_{j}} 
<\infty,\ \forall k\in \nc \Bigr\},\\
\Lambda_{1}(\alpha)
&=\Bigl\{ x=(x_{j})_{j\in \nc}\in \mathbb{K}^{\mathbb{N}} :
\|x\|_{k}
=\sum_{j=1}^{\infty} |x_{j}|e^{-\frac{1}{k}\alpha_{j}} 
<\infty,\ \forall k\in \nc \Bigr\}.
\end{align*}

\subsection{Operators on Fr\'echet Spaces}

Let $E$ and $F$ be Fr\'echet spaces. A linear operator $T:E\to F$ is \textit{continuous} if for every $k\in \nc$, there exist an $m_{k}\in \nc$ and a $C_{k}>0$ such that
$$\|T x\|_{k}\leq C_{k}\|x\|_{m_{k}}$$
for every $x\in E$. $L(E,F)$ denotes the space of all continuous linear operators from $E$ into $F$. An operator $T\in L(E,F)$ is \textit{bounded} if there exists a neighborhood $U$ of zero in $E$ such that $T(U)$ is a bounded set in $F$. $LB(E,F)$ denotes the space of all bounded linear operators from $E$ into $F$. 

Let $T \in L(E,F)$. We define
\[ \|T\|_{k,r} := \sup_{\|x\|_r\leq 1}\|Tx\|_k=\inf \Bigl\{ C:\|Tx\|_k\leq C\|x\|_r\; \text{for all} \; x\in E\Bigr\}\]
for every $k,r \in \mathbb{N}$. We note that $\|T\|_{k,r}$ need not be finite for every $k,r \in \mathbb{N}$ however, since  $T\in L(E,F)$, there exists a function $\phi:\mathbb{N} \to \mathbb{N}$ such that $\|T\|_{k,\phi(k)} < \infty$ for every $k$.

The \textit{characteristic of  continuity map} of a linear operator $T$ is the map $\pi_{T}:\nc\to\nc$ defined by
$$\pi_T(k)= \inf\lbrace r\in \nc: \sup_{\|x\|_r\leq 1}\|Tx\|_k<\infty \rbrace= \inf\{r \in \nc:\|T\|_{k,r}<\infty\}.$$
We remark that an operator $T$ is bounded if and only if its characteristic of continuity is bounded.

\subsection{Tameness}
In general, no a priori estimate is available for the characteristic of continuity of a continuous linear operator. The notion of tameness provides a way to control the potentially unpredictable behavior of the characteristic, as defined below:

\begin{defin}
Let $E$ and $F$ be Fr\'echet spaces, and $T \in L(E,F)$.   $T$ is called $S$-tame for a nondecreasing function $S:\mathbb{N}\to\mathbb{N}$ if there exists a $k_0 \in \mathbb{N}$ such that
$$\hspace{1.85in}\pi_T(k)\leq S(k),\hspace{1.45in} \forall k\geq k_{0}.$$ We call $T$ \emph{linearly tame} if $S$ is linear.
\end{defin}

Such behavior arises naturally in the context of operators between power series spaces. Vogt \cite{vogt5} established that every continuous linear operator between power series spaces of finite type is linearly tame. 

\begin{defin} Let $E$ and $F$ be Fr\'echet spaces. The pair $(E,F)$ is called tame, denoted by $(E,F)\in \mathfrak{T}$ if there exists a nondecreasing function $S:\mathbb{N}\to \mathbb{N}$ such that for every continuous linear operator $T\in L(E,F)$,  there exists $k_0 \in \mathbb{N}$ with
$$\hspace{1.85in}\pi_T(k)\leq S(k),\hspace{1.45in} \forall k\geq k_{0}.$$
If $E=F$, the space $E$ is called tame.
Moreover, if $\mathcal{U}\subset L(E,F)$ is a set of continuous linear operators and every $T\in \mathcal{U}$ satisfies the above condition, then we say that the set $\mathcal{U}$ is \emph{tame}.
\end{defin}

In \cite{ND3,ND4}, the second author showed that the families of Toeplitz operators between power series spaces are 
$S$-tame for an appropriate function 
$S$, and the families of Hankel operators are compact; consequently, both families are S-tame.

\begin{rem}
\begin{itemize}
\item[i.] The definition of tameness is independent of the choice of seminorms on $E$ and $F$.
\item[ii.] The tameness of a pair $(E,F)$ is also characterized by the following equivalent condition:  
there exists a sequence of nondecreasing functions $(S_\alpha)_{\alpha\in\mathbb{N}}$, such that for every $T\in L(E,F)$ there exists an $\alpha\in \nc$ such that
$$\hspace{1.8in}\pi_T(k)\leq S_{\alpha}(k),\hspace{1.45in} \forall k\geq k_{0}.$$
\end{itemize}
\end{rem}

\section{Tame Factorization Property} 
In this section, we introduce the tame factorization property
and discuss its basic properties.

Let $E$, $F$, and $G$ be Fréchet spaces. We say that a linear operator $T \in L(E,F)$ factors over $G$ if there exist operators $Q \in L(E,G)$ and $R \in L(G,F)$ such that $T = RQ$. The set of all operators that factor over $G$ will be denoted by $L^G(E,F)$.

\begin{defin} 
 Let $E$, $F$ and $G$ be Fr\'echet spaces. We say that the triple $(E,G,F)$ has tame factorization property, denoted by $(E,G,F)\in \mathfrak{TF}$, if there exists a nondecreasing function $S:\mathbb{N}\to \mathbb{N}$ such that for every $T
\in L^G(E,F)$, there exists $k_0 \in \mathbb{N}$ such that
$$\hspace{1.9in}\pi_T(k)\leq S(k)\hspace{1.4in}\forall k\geq k_{0}.$$  
\end{defin}

Equivalently, $(E,G,F)$ has the tame factorization property if there exists a sequence of nondecreasing functions $S_\alpha:\mathbb{N} \to \mathbb{N}$, $\alpha \in \mathbb{N}$, such that for every $T \in L^G(E,F)$, there exists an $\alpha$ satisfying
$$\hspace{1.9in}\pi_T(k)\leq S_{\alpha}(k)\hspace{1.4in}\forall k\in \nc.$$ 

\begin{rem}\label{Rem1}
Let $E$, $F$ and $G$ be Fr\'echet spaces. The following important statements are straightforward.
\begin{enumerate}
    \item If $(E,G)\in \mathfrak{T}$ and $(G,F)\in \mathfrak{T}$, then $(E,G,F)\in \mathfrak{TF}$.
    \item If $(E, G_{1})\in \mathfrak{T}$ and $(E, G_{2})\in \mathfrak{T}$, then $(E, G_{1}\times G_{2},F)\in \mathfrak{TF}$. 
    \item $(E,G,F)\in \mathfrak{TF}$ for any G, if $(E,F)\in \mathfrak{T}$.
    \item Let $X$ be a quotient of $E$, $Y$ a complemented subspace of $G$, and $Z$ a subspace of $F$.  
Then, if $(E,G,F) \in \mathfrak{TF}$, it follows that $(X,Y,Z) \in \mathfrak{TF}$.
\end{enumerate}
\end{rem}

\section{Characterization of the Tame Factorization Property}

In this section we characterize the tame factorization property for arbitrary Fr\'echet spaces by Theorem~4.2. The proof relies on a bilinear version of Grothendieck's factorization theorem (\cite[Int IV, Th\'eor\`eme A]{grot}, p. I-16) due to Terzio\u{g}lu and Zahariuta \cite[Lemma 2.1]{terzi2}. We then examine the rank-one case, Proposition~4.1, and specialize the resulting condition to triples involving K\"othe spaces (Propositions~4.2--4.5, Theorem~4.4), before turning to the behavior of the tame factorization property under projective tensor products (Theorems~4.5-- 4.6).

For a bilinear map $\theta:E\times F \to G$, the linear maps $\theta_x:F\to G$ and $\theta_y:E\to G$ are defined by
\[\hspace{1.45in}\theta(x,y)=\theta_x(y)=\theta_y(x),\quad\hspace{0.65in} (x,y)\in E\times F.\] 
The bilinear map $\theta$ is called \emph{sequentially separately closed} if $\theta_x$ and $\theta_y$ have sequentially closed graphs for any $x\in E$ and $y\in F$. 

\begin{lemma}[Bilinear Factorization Lemma {\cite[Lemma~2.1]{terzi2}}]\label{BFL}
	Let $E_1,E_2$ and $H_n$ be Fr\'echet spaces and $H=\bigcup H_n$ be an LF-space. If a bilinear map $\theta:E_1\times E_2\to H$ is sequentially separately closed, there is some $k_0\in \nc$ such that $\theta(E_1\times E_2)\subset H_{k_0}$ and $\theta:E_1\times E_2\to H_{k_0}$ is continuous.  
\end{lemma}
For a map $\phi:\nc\to \nc$, we define
\begin{align*}
L_\phi(E,F)
&= \Bigl\{ T\in L(E,F) : \|T\|_{k,\phi(k)}<\infty \text{ for all } k\in \mathbb{N} \Bigr\},\\
L_{\phi,n}(E,F)
&= \Bigl\{ T\in L(E,F) : \|T\|_{k,\phi(k)}<\infty \text{ for all } k\ge n \Bigr\}.
\end{align*}
The spaces $L_\phi(E,F)$ and $L_{\phi,n}(E,F)$ are Fr\'echet spaces  with respect to the seminorms
\[\hspace{1.5in}
|T|_m = \max_{1\le i\le m} \|T\|_{i,\phi(i)},\hspace{1.25in} m\ge 1,
\]
and
\[\hspace{1.5in}
|T|_{m,n} = \max_{n\le i\le m} \|T\|_{i,\phi(i)}, \hspace{1.15in} m\ge n,
\]
respectively, and we have:
\[
L_\phi(E,F)=\bigcap_{n\in \mathbb{N}} L_{\phi,n}(E,F).
\]

With this setup, we can start to examine tame factorization property in a general Fr\'echet space context.

\begin{theorem}\label{sec:thm2}
Let $E$, $F$, and $G$ be Fréchet spaces. 
Then the triple $(E,G,F)$ has the tame factorization property if and only if there exists a nondecreasing function 
$S:\mathbb{N}\to\mathbb{N}$ such that for every function $\phi:\mathbb{N}\to\mathbb{N}$ we have
\begin{equation}\label{sec:tamefactorineq}
\begin{split}
\exists k_0\in\mathbb{N}\quad &\forall k\ge k_0 \quad \exists q_k\in\mathbb{N}, C_k>0 \\
&\|T\|_{k,S(k)} 
\le C_k 
\max_{1\le q\le q_k} \|R\|_{q,\phi(q)} 
\max_{1\le q\le q_k} \|Q\|_{q,\phi(q)},
\end{split}
\end{equation}
for every $Q\in L_\phi(E,G)$, $R\in L_\phi(G,F)$ with $T=RQ$.  
\end{theorem}

\begin{proof} Let us assume that there exists a nondecreasing $S:\nc\to \nc$ such that for every function $\phi:\nc\to \nc$, the condition \ref{sec:tamefactorineq} is satisfied for every  $Q\in L_\phi(E,G)$, $R\in L_\phi(G,F)$ and $T=RQ$. 

Under this assumption, we show that the triple 
$(E,G, F)$ satisfies the tame factorization property. Let $T\in L^G(E,F)$ and $T=RQ$. 
Now we define a function $\varphi:\nc\to \nc$ as
$$\varphi(k)=\max\bigl\{\pi_R(k),\pi_Q(k)\bigr\} $$
for every $k\in \nc$. So we can write
\[ \|R\|_{q,\varphi(q)}=\sup_{\|x\|_{\varphi(q)}\leq 1}\|Rx\|_q\leq \sup_{\|x\|_{\pi_{R}(q)}\leq 1}\|Rx\|_q=\|R\|_{q,\pi_R(q)}<\infty , \]
and 
\[\|Q\|_{q,\varphi(q)}=\sup_{\|x\|_{\varphi(q)}\leq 1}\|Qx\|_q\leq \sup_{\|x\|_{\pi_{Q}(q)\leq 1}}\|Qx\|_q=\|Q\|_{q,\pi_Q(q)}<\infty \]
for every $q\in \mathbb{N}$. Let us apply condition \ref{sec:tamefactorineq} to the function $\varphi$. It then follows that there exists a $k_{0}\in\mathbb{N}$ such that for every $k\ge k_{0}$, there exist $q_{k}\in\mathbb{N}$ and $C_{k}>0$
satisfying 
\begin{equation*}
\begin{split}
\|T\|_{k,S(k)}&\leq C_{k} \max_{1\leq q\leq q_k}\Bigl\lbrace \|R\|_{q,\varphi(q)}\Bigr\rbrace \max_{1\leq q\leq q_k}\Bigl\lbrace \|Q\|_{q,\varphi(q)}\Bigr\rbrace
\\
&\leq C_k \max_{1\leq q\leq q_k}\Bigl\lbrace \|R\|_{q,\pi_{R}(q)}\Bigr\rbrace \max_{1\leq q\leq q_k}\Bigl\lbrace \|Q\|_{q,\pi_{Q}(q)}\Bigr\rbrace
<\infty.
\end{split}
\end{equation*}
So $T$ is S-tame. Since $T\in L^G(E,F)$ was arbitrary, this shows that $(E,G,F)$ has tame factorization property.

For the other direction, suppose that the triple $(E,G,F)$ has tame factorization property. Then there exists a map $S:\nc  \to\nc$ so that for every operator $T \in L^G(E,F)$, there exists a $k \in \mathbb{N}$ such that 
$T \in L_{S,k}(E,F)$. Consequently,
\[L^G(E,F)\subseteq \bigcup_{n\in \nc}L_{S,n}(E,F),\]
which is an LF-space.

Given $\phi:\nc\to\nc$, we define a bilinear map \[\theta:L_\phi(E,G)\times L_\phi(G,F)\to L^G(E,F) \quad \quad \theta(Q,R):=RQ.\] 
Since $\theta$ is sequentially separately closed, by Lemma \ref{BFL} there exists some $k_0\in \nc$ such that
\[\theta(L_\phi(E,G)\times L_\phi(G,F))\subseteq L_{S,k_0}(E,F), \]
and the associated canonical operator, where $\hat{\otimes}_\pi$ denotes the completion of the projective tensor product,
\[\hat{\theta}:L_\phi(E,G)\hat{\otimes}_\pi L_\phi (G,F)\to L_{S,k_0}(E,F) \]
 is linear and continuous. This means that for every $ k\geq k_0$ there exist $q_k\in \nc$ and $C_k>0$ such that 
\begin{align*}
	|\hat{\theta}(Q\otimes R)|_{k,k_0}&\leq C_{k}|R|_{q_k}|Q|_{q_k}\\
	\|RQ\|_{k,S(k)}\leq \max_{k_0\leq q\leq k} \|RQ\|_{q,S(q)} &\leq C_k\max_{1\leq q\leq q_k}\|R\|_{q,\phi(q)}\max_{1\leq q\leq q_k}\|Q\|_{q,\phi(q)}.
\end{align*} 
This gives the desired result. 
\end{proof}
\begin{rem}
If $E=G$ or $G=F$, then Theorem 4.2 coincides with Piszczek’s characterization of pairs of Fréchet spaces satisfying  $(E,F)\in \mathfrak{T}$ (see \cite[Theorem 2.1]{pisek2}).
\end{rem}

\begin{rem}
	We note that, for an arbitrary function $\phi:\nc\to \nc$, if we define:
	\[\psi:\nc\to\nc\quad \psi(n):=\max_{1\leq m\leq n}\phi(m),\]
	then $\psi(n)\geq \phi(n)$, so $\|x\|_{\psi(n)}\geq \|x\|_{\phi(n)}$, which implies $\|\cdot\|_{n,\psi(n)}\leq \|\cdot\|_{n,\phi(n)}$  for all $n\in \nc$. So without loss of generality, we can take each $\phi$  in Theorem \ref{sec:thm2} to be nondecreasing. 
\end{rem}

We now focus on rank one operators. Let $E$, $F$ and $G$ be Fr\'echet spaces and $u^{\prime}\in E^{\prime}$, $z\in F$. The rank one operator $T=u^{\prime}\otimes z$ defined as
\begin{equation*}
\begin{split}
T:\;&E\to F~\\
&x\to T(x)=(u^{\prime}\otimes z)(x)= u^{\prime}(x)z.
\end{split}
\end{equation*} 
It is easily seen that 
$T$ is linear and continuous. Now we choose some $w\in G$ and $w^{\prime}\in G^{'}$ satisfying $w^{\prime}(w)=1$, and define the operators
\[
Q:\;E\to G, \qquad Q(x)=(u^{\prime}\otimes w)(x)= u^{\prime}(x)w,
\]
and 
\[R:\;G\to F, \qquad R(y)=(w^{\prime}\otimes z)(y)= w^{\prime}(y)z.\]
Then, it follows that $T=RQ$ and $T$ belongs to $L^{G}(E,F)$. Furthermore, for all $k,r\in \mathbb{N}$, we have 
\[\|T\|_{k,r}= \|u^{\prime}\otimes z\|_{k,r}=\sup_{\|x\|_r\leq 1}|u^{\prime}(x)|\|z\|_k=\|u^{\prime}\|^*_r\|z\|_k.\]
Similarly, we can write
\[\|Q\|_{k,r}= \|u^{\prime}\|^*_r\|w\|_k, \qquad \text{and} \qquad \|R\|_{k,r}= \|w^{\prime}\|^*_r\|z\|_k.\]
It is obvious that for any $\phi:\nc\to \nc$, $Q\in L_\phi(E,G)$ and $R\in L_\phi (G,F)$. 

Applying Theorem~\ref{sec:thm2} to rank one operators, we obtain the following proposition.

\begin{pro}\label{sec:prop1}
 If $(E, G, F)$ has tame factorization property, then there exists a nondecreasing $ S:\mathbb{N}\to\mathbb{N}$ such that for every nondecreasing $\phi:\mathbb{N}\to\mathbb{N}$ we have 
\begin{equation*}
\begin{split}
\exists k_0\in\mathbb{N} \quad &\forall k\ge k_0,\quad \exists q_k\in\mathbb{N}, C_{k}>0 \\
&\|u^{\prime}\|^*_{S(k)}\|z\|_k\leq C_{k}\max_{1\leq q\leq q_k} \Bigl\lbrace \|u^{\prime}\|^*_{\phi(q)}\|w\|_q\Bigl\rbrace \max_{1\leq q\leq q_k}\Bigl\lbrace\|w^{\prime}\|^*_{\phi(q)}\|z\|_q\Bigl\rbrace 
\end{split}\end{equation*}
for every $u^{\prime}\in E^{\prime}$, $z\in F$,  $w^{\prime}\in G^{\prime}$ and $w\in G$, with $w^{\prime}(w)=1$. 
\end{pro}

We now want to examine what happens when some of the spaces under consideration are Köthe spaces. In some cases, the condition given in Proposition~\ref{sec:prop1} is not only necessary but also sufficient for the tame factorization property, which yields a more concrete characterization.

\begin{pro}\label{sec:cor1}
Let $\lambda(B)$ be a nuclear K\"othe space. $(E,\lambda(B),F)\in \mathfrak{TF}$ if and only if there exists a nondecreasing $ S:\mathbb{N}\to\mathbb{N}$ such that for every nondecreasing $\phi:\mathbb{N}\to\mathbb{N}$ we have
\begin{equation}\label{sec:wfac1}
\begin{split}
\exists k_0\in\mathbb{N} \quad &\forall k\ge k_0,\quad \exists q_k\in\mathbb{N}, C_{k}>0 ~\\	&\|u^{\prime}\|^*_{S(k)}\|z\|_k\leq C_k\max_{1\leq q\leq q_k} \Bigl\lbrace \|u^{\prime}\|^*_{\phi(q)}b_{i,q}\Bigl\rbrace \max_{1\leq q\leq q_k}\Biggl\lbrace\frac{\|z\|_q}{b_{i,\phi(q)}}\Biggr\rbrace
\end{split}
\end{equation}
for every  $u^{\prime}\in E^{\prime}$, $z\in F$, $i\in \mathbb{N}$. 
\end{pro}

\begin{proof} Let assume that  $(E,\lambda(B),F)\in \mathfrak{TF}$, $u^{\prime}\in E^{\prime}$, $z\in F$, and $i\in \mathbb{N}$. Let $e_{i}$ denote the sequence $(0,\dots,0,1,0,\dots) $ with the entry 1 in the $i$-th position and zeros elsewhere. For all $x\in \lambda(B)$ and $i\in \nc$, we define 
$e'_i(x)= x_i$. 
Since $e_{i}\in \lambda(B)$ and  $e^{\prime}_{i}\in (\lambda(B))^{\prime}$, we have 
$$ e_i'(e_{i})=1, \qquad \|e_i\|_k=b_{i,k}\qquad \text{and} \qquad \|e_i'\|^*_k=\frac{1}{b_{i,k}}$$
for every  $k\in \mathbb{N}$. By Proposition \ref{sec:prop1}, the condition \ref{sec:wfac1} holds, that is, there exists a map $ S:\mathbb{N}\to\mathbb{N}$ such that for every nondecreasing function $ \phi:\mathbb{N}\to\mathbb{N}$ we have \begin{equation*}
\begin{split}
\exists k_0\in\mathbb{N} \quad &\forall k\ge k_0,\quad \exists q_k\in\mathbb{N}, C_{k}>0 ~\\	&\|u^{\prime}\|^*_{S(k)}\|z\|_k\leq C_k\max_{1\leq q\leq q_k} \Bigl\lbrace \|u^{\prime}\|^*_{\phi(q)}b_{i,q}\Bigl\rbrace \max_{1\leq q\leq q_k}\Biggl\lbrace\frac{\|z\|_q}{b_{i,\phi(q)}}\Biggr\rbrace
\end{split}
\end{equation*}
for every  $u^{\prime}\in E^{\prime}$, $z\in F$, $i\in \mathbb{N}$. This establishes one direction of the proof.

Now we suppose that there exists a nondecreasing function $\exists S:\mathbb{N}\to\mathbb{N}$ such that for every nondecreasing function $\phi:\mathbb{N}\to\mathbb{N}$, the condition \ref{sec:wfac1} holds
for every  $u^{\prime}\in E^{\prime}$, $z\in F$, $i\in \mathbb{N}$. 

Since $\lambda(B)$ is nuclear, Grothendieck-Pietsch Criteria (\cite{vogt4}[Theorem 28.15]) yield that given $\phi:\nc\to\nc$, there exists a $ \psi:\nc\to\nc$ satisfying $\psi(k)> \phi(k)$  and 
\[D(k)=\sum_{i=1}^\infty\frac{b_{i,\phi(k)}}{b_{i,\psi(k)}}<\infty\]
for every $k\in \nc$. Without loss of generality, we may assume that $\psi$ is nondecreasing.

Let $T\in L(E,F)$ be an operator that admits a factorization $T=RQ$ through $\lambda(B)$. For every $i\in \nc$, we define 
$$u'_i=e_i'\circ Q\in E' \qquad \text{and} \qquad z_i=Re_i\in F.$$ 
Then we have
$$ Qx=\sum^{\infty}_{i=1} u'_{i}(x)e_{i} \qquad\text{and} \qquad  Tx=\sum_{i=1}^\infty u'_i(x)z_i.$$
By our assumption, there exists a $k_{0}\in \nc$ such that for every $k\geq k_{0}$, there exist $q_k\in \nc$ and $C_k>0$ such that the inequality
\begin{equation*}\label{sec:eqn1}
\|u'_i\|^*_{S(k)}\|z_i\|_k\leq C_k\max_{1\leq q\leq q_k} \Bigl\lbrace \|u'_i\|^*_{\psi(q)}b_{i,q}\Bigr\rbrace \max_{1\leq q\leq q_k}\Biggl\lbrace\frac{\|z_i\|_q}{b_{i,\psi(q)}}\Biggr\rbrace \end{equation*}
holds for every $i\in \mathbb{N}$. Since $\psi(k)> \phi(k)$ for every $k\in \mathbb{N}$, we write
\[\|u'_i\|^*_{\psi(q)}b_{i,q}\leq \|u'_i\|^*_{\phi(q)}b_{i,q} =\sup_{\|x\|_{\phi(q)}\leq 1}|u'_i(x)|\|e_i\|_q\leq \|Q\|_{q,\phi(q)},\]
and since $B$ is a K\"{o}the matrix, we write
\begin{equation*}
\begin{split}
\sum_{i=1}^\infty \frac{\|z_i\|_q}{b_{i,\psi(q)}}=\sum_{i=1}^\infty \frac{\|z_i\|_q}{b_{i,\phi(q)}} \frac{b_{i,\phi(q)}}{b_{i,\psi(q)}}\leq D(q)\sup_{i\in\nc} \frac{\|z_i\|_q}{b_{i,\phi(q)}} \leq D(q)\|R\|_{q,\phi(q)}. 
\end{split}
\end{equation*}
Combining these, we get:
\begin{align*}
\|T\|_{k,S(k)}&\leq \sum_{i=1}^\infty \|u'_i\|^*_{S(k)}\|z_i\|_k\\
&\leq C_k\sum_{i=1}^\infty \max_{1\leq q\leq q_k} \Bigl\lbrace \|u'_i\|^*_{\psi(q)}b_{i,q}\Bigr\rbrace \max_{1\leq q\leq q_k}\Biggl\lbrace\frac{\|z_i\|_q}{b_{i,\psi(q)}}\Biggr\rbrace\\
&\leq C_k\max_{1\leq q\leq q_k}\|Q\|_{q,\phi(q)}\sum_{i=1}^\infty  \max_{1\leq q\leq q_k}\Biggl\lbrace\frac{\|z_i\|_q}{b_{i,\psi(q)}}
\Biggr\rbrace\\
&\leq C_k\Big(\sum_{q=1}^{q_k}D(q)\Big)\max_{1\leq q\leq q_k}\Bigl\lbrace\|Q\|_{q,\phi(q)}\Bigl\rbrace\max_{1\leq q\leq q_k}\Bigl\lbrace\|R\|_{q,\phi(q)}\Bigl \rbrace<\infty.
\end{align*}
By Theorem \ref{sec:thm2}, the triple $(E,\lambda(B),F)$ has tame factorization property.
\end{proof}

The following lemma provides explicit formulas for the operator seminorms that will be needed in the subsequent results.

\begin{lemma}\label{LE1} Let $r,p\in\mathbb{N}$. The operator seminorm $\|\cdot\|_{r,p}$ admits the following formulas in the cases considered below.
\begin{enumerate}
    \item  Let $Y$ be a Fr\'echet space and $S\in L(\lambda(A), Y)$. Then,
\begin{equation*}
\|S\|_{r,p}= \sup_{i\in \mathbb{N}}\Biggl\lbrace \frac{\|Se_{i}\|_{r}}{a_{i,p}}\Biggr\rbrace .
\end{equation*}
\item Let $U\in L(\lambda(A), \lambda(B))$ and $(u_{i,j})_{i,j\in \nc}$ be the matrix representation of $U$, defined by $\displaystyle U(e_{i})=\sum^{\infty}_{j=1} u_{i,j}e_{j}$ for every $i\in\nc$. Then,
$\vspace{-0.6em}$
\begin{equation*}
\|U\|_{r,p}= \sup_{i\in \nc} \Biggl\lbrace\sum^{\infty}_{j=1} |u_{i,j}| \frac{b_{j,r}}{a_{i,p}}\Biggr\rbrace
\end{equation*}
\item Let $V\in L(\lambda^{\infty}_{0}(A), \lambda(B))$ and $(v_{i,j})_{i,j\in \nc}$ be the matrix representation of $V$, given by $\displaystyle V(e_{i})=\sum^{\infty}_{j=1} v_{i,j}e_{j}$ for every $i\in\nc$. Then,
$\vspace{-0.6em}$
\begin{equation*}
\|V\|_{r,p}=\sum^{\infty}_{i=1} \sum^{\infty}_{j=1} |v_{i,j}| \frac{b_{i,q}}{a_{j,q}}.
\end{equation*}
\item Let $W\in L(\lambda^{\infty}_{0}(A), \lambda^{\infty}(B))$ and $(w_{i,j})_{i,j\in \nc}$ be the matrix representation of $W$, that is, $\displaystyle W(e_{i})=\sum^{\infty}_{j=1} w_{i,j}e_{j}$ for every $i\in\nc$. Then,
$\vspace{-0.6em}$
\begin{equation*}
\|W\|_{r,p}=\sup_{i\in \mathbb{N}} \sum^{\infty}_{j=1} |w_{i,j}| \frac{b_{i,r}}{a_{j,p}}
\end{equation*}
\end{enumerate}
\end{lemma}

\begin{proof}  
We prove the assertions separately.

\smallskip
\noindent
(1)
For every \(x = (x_i)_{i\in \nc} \in \lambda(A)\), we obtain
\begin{equation*}
\begin{split}
\|Sx\|_{r}= \Biggl\| \sum^{\infty}_{i=1} x_{i}Se_i\Biggr\|_{r}\leq \sum^{\infty}_{i=1} |x_i| \|Se_{i}\|_{r} =\sum^{\infty}_{i=1} |x_{i}| a_{i,p} \frac{\|Se_{i}\|_{r}}{a_{i,p}}  \leq \sup_{i\in \mathbb{N}}\Biggl\lbrace \frac{\|Se_{i}\|_{r}}{a_{i,p}}\Biggr\rbrace \|x\|_{p}.
 \end{split}
\end{equation*}
and hence
\begin{equation}\label{AR1}
\|S\|_{r,p}=\sup_{\|x\|_{p}\leq 1} \|Sx\|_{r} \leq \sup_{i\in \mathbb{N}}\Biggl\lbrace \frac{\|Se_{i}\|_{r}}{a_{i,p}}\Biggr\rbrace.
\end{equation}
Moreover, for every $i\in \nc$, we have $\|e_{i}\|_{p}=a_{ip}$, and therefore
$$\|S\|_{r,p}=\sup_{\|x\|_{p}\leq 1} \|Sx\|_{r} \geq \Biggl\| S\left( \frac{e_{i}}{a_{i,p}}\right)\Biggr\|_{r}=\frac{\|Se_{i}\|_{r}}{a_{i,p}}.$$ Taking the supremum over $i\in \nc$ yields 
\begin{equation}\label{AR2}
\|S\|_{r,p} \geq \sup_{i\in \mathbb{N}}\Biggl\lbrace \frac{\|Se_{i}\|_{r}}{a_{i,p}}\Biggr\rbrace.
\end{equation}
Combining \eqref{AR1} and \eqref{AR2} completes the proof of (1).

\smallskip
\noindent
(2) For each $i\in \nc$, $\displaystyle Ue_{i}=\sum^{\infty}_{j=1} u_{i,j}e_{j}$ and $ \displaystyle 
\|Ue_{i}\|_{r}=\sum^{\infty}_{j=1}|u_{i,j}|b_{j,r}$. By applying part (1) to the operator U, we obtain
\begin{equation*}
\|U\|_{r,p}=\sup_{i\in \mathbb{N}}\Biggl\lbrace \frac{\|Ue_{i}\|_{r}}{a_{i,p}}\Biggr\rbrace=\sup_{i\in \nc} \Biggl\lbrace\sum^{\infty}_{j=1} |u_{i,j}| \frac{b_{j,r}}{a_{i,p}}\Biggr\rbrace.
\end{equation*}
which proves (2).

\smallskip
\noindent
(3) For every $\displaystyle x=\sum^{\infty}_{i=1} x_ie_i\in \lambda^{\infty}_{0}(A)$, we have $\vspace{-0.4em}$
\begin{equation*}
\begin{split}
Vx=\sum^{\infty}_{i=1} x_{i}Ve_{i}=\sum^{\infty}_{i=1} x_{i} \sum^{\infty}_{j=1} v_{i,j}e_{j}=\sum^{\infty}_{j=1}\left(\sum^{\infty}_{i=1}v_{i,j}x_{i}\right)e_j
\end{split}
\end{equation*}
and $\vspace{-0.4em}$
\begin{equation*}
\begin{split}
\|V\|_{r,p}&=\sup_{\|x\|_{p}\leq 1} \|Vx\|_{r} =\sup_{\|x\|_{p}\leq 1} \sum^{\infty}_{i=1} |(Vx)_i|b_{ir} \\
&=\sup_{\|x\|_{p}\leq 1} \sum^{\infty}_{i=1} \Bigl| \sum^{\infty}_{j=1} v_{i,j}x_{j}\Bigr|b_{ir}= \sum^{\infty}_{i=1} \sum^{\infty}_{j=1} |v_{ij}| \frac{b_{i,r}}{a_{j,p}}.
\end{split} 
\end{equation*}
It follows that
\begin{equation}\label{AQ11}
\|V\|_{r,p}=\sum^{\infty}_{i=1} \sum^{\infty}_{j=1} |v_{i,j}| \frac{b_{i,r}}{a_{j,p}}.
\end{equation}

\smallskip
\noindent
(4)
The proof of (4) follows the same lines as the proof of (3). Since $(w_{i,j})_{i,j\in \nc}$ is the matrix representation of $W$, $$Wx=\sum^{\infty}_{j=1}\left(\sum^{\infty}_{i=1}v_{i,j}x_{i}\right)e_j$$ 
for every $x\in \lambda^{\infty}_{0}(A)$. Then we write 
\begin{equation*}
\begin{split}
\|W\|_{r,p}&=\sup_{\|x\|_{p}\leq 1} \|Wx\|_{r} =\sup_{\|x\|_{p}\leq 1} \sup_{i\in \mathbb{N}} |(Wx)_i|b_{i,r} \\
&=\sup_{\|x\|_{p}\leq 1} \sup_{i\in \mathbb{N}} \Bigl| \sum^{\infty}_{j=1} w_{i,j}x_{j}\Bigr|b_{i,r}=\sup_{i\in \mathbb{N}} \sum^{\infty}_{j=1} |w_{i,j}| \frac{b_{i,r}}{a_{j,p}}
\end{split} 
\end{equation*}
that is,
\begin{equation*}
\|W\|_{r,p}=\sup_{i\in \mathbb{N}} \sum^{\infty}_{j=1} |w_{i,j}| \frac{b_{i,r}}{a_{j,p}}
\end{equation*}
This completes the proof.
\end{proof}

We now obtain an equivalent characterization of the tame factorization property when the initial space is $\lambda(A)$ or $\lambda^{\infty}(A)$, without imposing nuclearity on $\lambda(B)$.

\begin{pro}\label{sec:cor2}
Let $F$ be a Fr\'echet space. Then the following statements are equivalent.
\begin{enumerate}
\item $(\lambda(A),\lambda (B),F)\in \mathfrak{TF}$,
\item $(\lambda^\infty_0(A),\lambda (B),F)\in \mathfrak{TF}$,
\item There exists a nondecreasing function $S:\mathbb{N}\to\mathbb{N}$ such that for every nondecrasing $\phi:\mathbb{N}\to\mathbb{N}$ we have 
\begin{gather}\label{sec:wfac2}
\begin{split}
\exists k_0\in\mathbb{N} \quad &\forall k\ge k_0,\quad \exists q_k\in\mathbb{N}, C_{k}>0 ~\\	&
\frac{\|z\|_k}{a_{i,S(k)}}\leq C_k\max_{1\leq q\leq q_k} \Biggl\lbrace \frac{b_{jq}}{a_{i,\phi(q)}}\Biggr\rbrace\max_{1\leq q\leq q_k}\Biggl\lbrace\frac{\|z\|_q}{b_{j,\phi(q)}}\Biggr\rbrace
\end{split}\end{gather}
for every $z\in F$, $i,j\in \nc$.
\end{enumerate}
\end{pro}

\begin{proof}
The implications $(1)\Rightarrow (3)$ and $(2)\Rightarrow (3)$ follow from Proposition \ref{sec:prop1} by taking $u'=e_j'$ for every $j\in \nc$. 

Next, we will prove the implication $(3)\Rightarrow (1)$. Suppose (3) holds. Then there exists a nondecreasing function $S:\mathbb{N}\to\mathbb{N}$ such that for every nondecreasing function $\phi:\mathbb{N}\to\mathbb{N}$ the condition \ref{sec:wfac2} holds for every $z\in F$, $i,j\in \nc$. 

To show that $(\lambda(A), \lambda(B), F)$ has tame factorization property, we will apply Theorem \ref{sec:thm2}.  
 
Let $\phi:\nc \to \nc$ be an arbitrary function and $T=RQ\in L^{\lambda(B)}(\lambda(A),F)$. 

Let $(u_{i,j})_{i,j\in \nc}$ be the matrix representation of the operator $Q\in (\lambda(A),\lambda(B))$, let $Re_j=z_j$ for all $j\in \nc$. Then  we have
\[Qe_i=\sum_{j\in\nc}u_{i,j}e_j\quad \textrm{and \quad}Te_i=RQe_i=\sum_{j\in\nc} u_{i,j}z_j.\]
Using Lemma \ref{LE1} we can write:
\begin{align}
	\|Q\|_{q,\phi(q)}&= \sup_{i\in \nc} \Biggl\lbrace\sum_{j\in\nc} |u_{i,j}| \frac{b_{j,q}}{a_{i,\phi(q)}}\Biggr\rbrace \label{A1}\\
	\|R\|_{q,\phi(q)}&= \sup_{j\in \mathbb{N}}\Biggl\lbrace \frac{\|Re_j\|_{q}}{b_{j,\phi(q)}}\Biggr\rbrace =\sup_{j\in \mathbb{N}}\Biggl\lbrace \frac{\|z_{j}\|_{q}}{b_{j,\phi(q)}}\Biggr\rbrace \label{A2}\\
	\|T\|_{k,S(k)}&=\sup_{i\in \mathbb{N}}\Biggl\lbrace \frac{\|Te_i\|_{k}}{a_{i,S(k)}}\Biggr\rbrace=\sup_{i\in \nc} \Biggl\lbrace\Big\|\sum_{j\in \nc} u_{i,j} \frac{z_{j}}{a_{i,S(k)}}\Big\|_k\Biggr\rbrace \label{A3}
\end{align}
By the assumption (3) and norms in \eqref{A1}, \eqref{A2}, and \eqref{A3}, we obtain
\begin{align*}
	\|T\|_{k,S(k)} & =\sup_{i\in \nc} \Biggl\lbrace\Big\|\sum_{j\in \nc} u_{i,j} \frac{z_{j}}{a_{i,S(k)}}\Big\|_k\Biggr\rbrace \leq \sup_{i\in \nc} \Biggl\lbrace\sum_{j\in \nc} |u_{i,j}| \frac{\|z_{j}\|_k}{a_{i,S(k)}}\Biggr\rbrace\\
	& \leq  C_k \sup_{i\in\nc} \Biggl\lbrace \sum_{j\in\nc}|u_{i,j}|\max_{1\leq q\leq q_k}\Bigl\lbrace\frac{b_{j,q}}{a_{i,\phi(q)}}\Bigr\rbrace \max_{1\leq q\leq q_k} \Bigl\lbrace\frac{\|z_j\|_q}{b_{j,\phi(q)}}\Bigr\rbrace\Biggr\rbrace\\
	& \leq C_k\max_{1\leq q\leq q_k} \|R\|_{q,\phi(q)}\max_{1 \leq q\leq q_k} \Biggl\lbrace \sup_{i\in \mathbb{N}}\sum_{i\in\nc} |u_{i,j}|\frac{b_{j,q}}{a_{i,\phi(q)}} \Biggr\rbrace \\ 
	& \leq  C_k \max_{1\leq q\leq q_k} \|R\|_{q,\phi(q)}\max_{1\leq q\leq q_k} \|Q\|_{q,\phi(q)} 
\end{align*}
Therefore $(\lambda(A),\lambda (B),F)\in \mathfrak{TF}$.

Finally, we prove that $(3) \Rightarrow (2)$. In this step, we essentially follow the same approach as above. Suppose (3) holds.  Then there exists a nondecreasing function $S:\mathbb{N}\to\mathbb{N}$ such that for every nondecreasing function $\phi:\mathbb{N}\to\mathbb{N}$ the condition \ref{sec:wfac2} holds for $z\in F$, $i,j\in \nc$. 

To show that $(\lambda^{\infty}_{0}(A), \lambda(B), F)$ has tame factorization property, we will apply Theorem \ref{sec:thm2}. 

Let $\phi:\nc \to \nc$ be an arbitrary function and $T=RQ\in L^{\lambda(B)}(\lambda^\infty_0(A),F)$. 

Let $(u_{i,j})_{i,j\in \nc}$ be the matrix representation of the operator $Q\in (\lambda^\infty_0(A),\lambda(B))$, let $Re_j=z_j$ for all $j\in \nc$. Then $\displaystyle Qe_i=\sum_{j\in\nc}u_{i,j}e_j$ and we have:
\[
Tx=\sum_{i\in\nc}x_iTe_i=\sum_{i\in\nc}\sum_{j\in\nc}x_iu_{i,j}z_j.
\]
Using Lemma \ref{LE1} we can write:
\begin{align}
	\|Q\|_{q,\phi(q)}&= \sum_{i\in \nc} \sum_{j\in\nc} |u_{i,j}| \frac{b_{j,q}}{a_{i,\phi(q)}} \label{B1}\\
	\|R\|_{q,\phi(q)}&= \sup_{j\in \mathbb{N}}\Biggl\lbrace \frac{\|Re_j\|_{q}}{b_{j,\phi(q)}}\Biggr\rbrace =\sup_{j\in \mathbb{N}}\Biggl\lbrace \frac{\|z_{j}\|_{q}}{b_{j,\phi(q)}}\Biggr\rbrace \label{B2}
\end{align}
Since $\displaystyle \|x\|_r=\sup_{i\in\nc}|x_i|a_{i,r}<\infty$ and $\displaystyle \lim_{i\to\infty}|x_i|a_{i,r}=0$ for every $x=(x_{i})_{i\in \nc}\in \lambda^{\infty}_{0}(A)$ and  every $r\in \nc$, we have:
\begin{align*}
	\|Tx\|_{k}&  \leq \sum_{i\in\nc} \sum_{j\in\nc} |x_{i}||u_{i,j}| \|z_j\|_{k} = \sum_{i\in\nc}\sum_{j\in\nc}  |x_i|a_{i,S(k)}|u_{ij}| \frac{\|z_j\|_{k}}{a_{i,S(k)}} \\
	& \leq  \|x\|_{S(k)} \sum_{i\in\nc}\sum_{j\in\nc}  |u_{i,j}| \frac{\|z_j\|_{k}}{a_{i,S(k)}}.
\end{align*}
Using this along with our assumption and the norms in \ref{B1} and \ref{B2} we get:
\begin{align*}
	\|T\|_{k,S(k)} & \leq  \sum^{\infty}_{i=1} \sum_{j=1}^\infty |u_{i,j}|\frac{\|z_j\|_k}{a_{i,S(k)}}\\
	& \leq  C_k \sum^{\infty}_{i=1}  \sum_{j=1}^\infty |u_{i,j}|\max_{1\leq q\leq q_k}\Bigl\lbrace\frac{b_{j,q}}{a_{i,\phi(q)}}\Bigr\rbrace \max_{1\leq q\leq q_k} \Bigl\lbrace\frac{\|z_j\|_q}{b_{j,\phi(q)}}\Bigr\rbrace\\
	& \leq C_k \max_{1\leq q\leq q_{k}} \sup_{j\in \nc} \Biggl\lbrace\frac{\|z_j\|_q}{b_{j,\phi(q)}}\Biggr\rbrace  \max_{1 \leq q\leq q_k} \Biggl\lbrace \sum^{\infty}_{i=1}\sum_{j=1}^\infty |u_{i,j}|\frac{b_{j,q}}{a_{i,\phi(q)}} \Biggr\rbrace \\ 
	& \leq  C_k \max_{1\leq q\leq q_k} \|R\|_{q,\phi(q)} \max_{1\leq q\leq q_k} \|Q\|_{q,\phi(q)}
\end{align*}
Therefore by Theorem \ref{sec:thm2},  $(\lambda^{\infty}_{0}(A),\lambda (B),F)$ has tame factorization property.
\end{proof}

We now consider the case where the initial space $E$ is an arbitrary Fréchet space and the remaining spaces are Köthe spaces.

\begin{pro} Let $E$ be a Fr\'echet space. Then the following statements are equivalent: 
	\begin{enumerate}
		\item $(E,\lambda^\infty_0(B),\lambda (C))\in \mathfrak{TF}$
		\item $(E,\lambda^\infty_0(B),\lambda^\infty (C))\in \mathfrak{TF}$
		\item  There exists a nondecreasing function $S:\mathbb{N}\to\mathbb{N}$ such that for every nondecrasing $\phi:\mathbb{N}\to\mathbb{N}$ we have
\begin{gather}\label{sec:wfac3}
\begin{split}
\exists k_0\in\mathbb{N} &\quad \forall k\ge k_0,\quad \exists q_k\in\mathbb{N}, C_{k}>0 ~\\	&
	\|u'\|^*_{S(k)}c_{j,k}\leq C_k\max_{1\leq q\leq q_k}\Biggl\lbrace \frac{c_{j,q}}{b_{i,\phi(q)}}\Biggr\rbrace\max_{1\leq q\leq q_k}\Bigl\lbrace b_{i,q}\|u^{\prime}\|^*_{\phi(q)}\Bigr\rbrace
\end{split}
\end{gather}
for every $u'\in E'$, $i,j\in \nc$. 
	\end{enumerate}
\end{pro}

\begin{proof} The implications $(1)\Rightarrow (3)$ and $(2)\Rightarrow (3)$ follow from Proposition~\ref{sec:prop1} for every $u'\in E^{\prime}$ and $z=e_{j}$ for every $j\in \nc$. 
	
We will now establish the implication  $(3)\Rightarrow (1)$. 
The assumption (3) gives us that there exists a nondecreasing function $S:\mathbb{N}\to\mathbb{N}$ such that  for every nondecreasing function  $\phi:\mathbb{N}\to\mathbb{N}$ 
the condition \ref{sec:wfac3} holds
for every $u'\in E'$, $i,j\in \nc$. 

We show that the same function $S:\mathbb{N}\to\mathbb{N}$ has the following property: for every function $\phi:\mathbb{N}\to\mathbb{N}$ there exists $k_0\in\mathbb{N}$ such that, for all $k\ge k_0$, one can find $q_k\in\mathbb{N}$ and $C_k>0$ for which
\begin{equation}\label{sec:tamefactorineq}
	\|T\|_{k,S(k)} 
	\le C_k 
	\max_{1\le q\le q_k} \|R\|_{q,\phi(q)} 
	\max_{1\le q\le q_k} \|Q\|_{q,\phi(q)},
\end{equation}
holds whenever $T=RQ\in L^G(E,F)$.

Let $\phi:\nc \to \nc$ be an arbitrary function and let $T\in L(E,\lambda(C))$  be an operator that admits a factorization $T=RQ$ for some $Q\in L(E,\lambda^{\infty}_{0}(B))$ and  $R\in L(\lambda^{\infty}_{0}(B), \lambda(C))$. Let $(w_{i,j})_{i,j\in \nc}$ be the matrix representation of the operator $R$. For every $i\in \nc$, we define $u_i'\in E'$ by $u'_i:=e_i'\circ Q$. Since $e'_i(y)=y_i$ for any $y=(y_i)_{i\in\nc}$, we have:
\begin{gather}\label{ops1}
	Qx=(u'_i(x))_{i\in\nc}=\sum_{i\in\nc} u'_{i}(x)e_{i}, \quad 	Ry=\sum_{i\in\nc}\sum_{j\in\nc}y_iw_{i,j}e_j 
\end{gather}
for $x\in E$, $y=(y_i)_{i\in\nc}\in\lambda^\infty_0(B)$. Using Lemma \ref{LE1}, we can write the norms
\begin{align*}
	\|Q\|_{q,\phi(q)}&=\sup_{\|x\|_{\phi(q)}\leq 1} \sup_{i\in \mathbb{N}} |u'_i(x)|b_{i,q} =\sup_{i\in\nc}\|u'_i\|^*_{\phi(q)}b_{i,q}\\
	\|R\|_{q,\phi(q)}&= \sum_{i\in \nc} \sum_{j\in\nc} |w_{i,j}| \frac{c_{iq}}{b_{j\phi(q)}}
\end{align*}
Since $T=RQ$, we have:
\begin{gather}\label{ops2}Tx=\sum_{i\in\nc} u'_{i}(x)R(e_{i})=\sum_{i\in\nc}\sum_{j\in\nc} u'_i(x)w_{i,j}e_j\end{gather}
for every $x\in E$. For every $\|x\|_{S(k)}\leq 1$ we get:
\begin{align*}
	\|Tx\|_{k,S(k)} & \leq \sup_{\|x\|_{S(k)}\leq 1}\sum_{i\in\nc}\sum_{j\in\nc} |u'_i(x)||w_{i,j}|c_{jk}\leq  \sum_{i\in\nc}\sum_{j\in\nc} \|u'_i\|^*_{S(k)}|w_{i,j}|c_{j,k}\\
	& \leq  C_k\sum_{i\in\nc} \sum_{j\in\nc}|w_{i,j}|\max_{1\leq q\leq q_k}\Bigl\lbrace \frac{c_{j,q}}{b_{i,\phi(q)}}\Bigl\rbrace\max_{1\leq q\leq q_k}\Bigl\lbrace b_{i,q}\|u'_i\|^*_{\phi(q)}\Bigr\rbrace \\
	& \leq  C_k\max_{1\leq q\leq q_k}\sup_{i\in \nc} \Bigl\lbrace b_{i,q}\|u'_i\|^*_{\phi(q)}\Bigr\rbrace  \max_{1\leq q\leq q_k}\sum_{j=1}^\infty \sum_{i=1}^\infty|w_{i,j}|\frac{c_{j,q}}{b_{i,\phi(q)}}	\\
	& \leq  C_k\max_{1\leq q\leq q_k}\|R\|_{q,\phi(q)}\max_{1\leq q\leq q_{k}}  \|Q\|_{q,\phi(q)}
\end{align*}
Theorem \ref{sec:thm2} implies that $(E,\lambda^\infty_0(B),\lambda (C))$ has tame factorization property.

Finally, we prove that $(3) \Rightarrow (2)$. In this case, we essentially follow the same approach as above.
The assumption (3) gives us that there exists a nondecreasing function $S:\mathbb{N}\to\mathbb{N}$ such that  for every nondecreasing function  $\phi:\mathbb{N}\to\mathbb{N}$ 
the condition \ref{sec:wfac3} holds
for every $u'\in E'$, $i,j\in \nc$. 

Let $\phi:\nc \to \nc$ be an arbitrary function and let $T\in L(E, \lambda(C))$ be an operator satisfying $T=RQ$ for some $Q\in L(E,\lambda^{\infty}_{0}(B))$ and  $R\in L(\lambda^{\infty}_{0}(B),\lambda^{\infty}(C))$. Let $(w_{i,j})_{i,j\in \nc}$ be the matrix representation of the operator $R$, we define $u_i'\in E'$ by $u'_i:=e_i'\circ Q$ for every $i\in \nc$. Proceeding as above, for $x\in E$ and $y\in \lambda^\infty_0(B)$ we have
\begin{gather*}
	Qx=\sum_{i\in\nc} u'_{i}(x)e_{i}, \quad 	Ry=\sum_{i\in\nc}\sum_{j\in\nc}y_iw_{i,j}e_j, \quad Tx=\sum_{i\in\nc}\sum_{j\in\nc} u'_i(x)w_{i,j}e_j.
\end{gather*}
Using Lemma \ref{LE1} the following
\begin{gather*}
	\|R\|_{q,\phi(q)}=\sup_{i\in \mathbb{N}} \sum^{\infty}_{j=1} |w_{i,j}| \frac{c_{i,q}}{b_{j,\phi(q)}},\quad \|Q\|_{q,\phi(q)}=\sup_{i\in\nc} \Bigl\lbrace b_{i,q}\|u'_i\|^{*}_{\phi(q)}\Bigr\rbrace.
\end{gather*}
Now we can estimate the norm of $T$: 
\begin{align*}
	\|T\|_{k,S(k)}&\leq \sup_{\|x\|_{S(k)}\leq 1}\Big(\sup_{j\in \nc} \sum_{i\in\nc} |u'_i(x)||w_{i,j}|c_{j,k}\Big) =  \sup_{j\in \nc} \sum_{i\in\nc} \|u'_i\|^{*}_{S(k)}|w_{i,j}|c_{j,k} \\
	& \leq C_k \sup_{j\in \nc} \sum_{i\in\nc} |w_{i,j}|\max_{1\leq q\leq q_{k}}\Bigl\lbrace \frac{c_{j,q}}{b_{i,\phi(q)}}\Bigr\rbrace \max_{1\leq q\leq q_{k}} \Bigl\lbrace b_{i,q}\|u'_{i}\|^{*}_{\phi(q)}\Bigr\rbrace  \\
	& \leq C_k\max_{1\leq q\leq q_{k}} \left(\sup_{i\in \mathbb{N}}\Bigl\lbrace b_{i,q}\|u'_{i}\|^{*}_{\phi(q)}\Bigr\rbrace\right) \max_{1\leq q\leq q_{k}} \sup_{j\in \nc}\sum_{i\in\nc} |w_{i,j}|\frac{c_{j,q}}{b_{i,\phi(q)}} \\
	&=C_k \max_{1\leq q\leq q_{k}} \|Q\|_{q,\phi(q)} \max_{1\leq q\leq q_{k}}\|R\|_{q,\phi(q)}.
\end{align*}
So $(E,\lambda^{\infty}_{0}(B),\lambda(C))$ has tame factorization property by Theorem \ref{sec:thm2}.
\end{proof}

We now characterize the tame factorization property for operators between Köthe spaces factoring through an arbitrary Fréchet space.

\begin{pro} Let $G$ be a Fr\'echet space.
$(\lambda(A),G,\lambda^\infty(C))\in \mathfrak{TF}$ if and only if there exists a nondecreasing function $S:\mathbb{N}\to\mathbb{N}$ such that for every nondecrasing $\phi:\mathbb{N}\to\mathbb{N}$ we have
\begin{equation}\label{P5}
\begin{split}
\exists k_0\in\mathbb{N} &\quad \forall k\ge k_0,\quad \exists q_k\in\mathbb{N}, C_{k}>0 ~\\	&
\frac{c_{j,k}}{a_{i,S(k)}}|y'(y)|\leq C_k \max_{1\leq q\leq q_k}\Biggl\lbrace\frac{\|y\|_q}{a_{i,\phi(q)}}\Biggr\rbrace \max_{1\leq q\leq q_k} \Bigl\lbrace \|y'\|^*_{\phi(q)}c_{j,q}\Bigr\rbrace
\end{split}
\end{equation}
for every $y\in G$, $y'\in G'$, $i,j\in \nc$.
\end{pro}

\begin{proof} Let us assume that $(\lambda(A), G, \lambda^{\infty}(C)) \in \mathcal{TF}$ and $i,j \in \mathbb{N}$, $y \in G$ and $y' \in E'$. We define the operators
\[Q= e'_j \otimes y,\qquad R = y' \otimes e_i, \qquad \text{and} \qquad T=RQ.
\]
It is obvious that  
$T\in L(\lambda(A),\lambda^{\infty}(C))$, $R\in L(G,\lambda^{\infty}(C))$, and $Q\in L( \lambda(A),G)$. For every $x\in \lambda(A)$, we have
$$Tx=RQx=R(e'_{j}(x)y)=e'_{j}(x)R(y)=e'_{j}(x)y'(y)e_{i}$$
and 
\begin{equation*}
\begin{split}
\|T\|_{k,S(k)}&=\sup_{\|x\|_{S(k)}\leq 1}\|Tx\|_{k}=\sup_{\|x\|_{S(k)}\leq 1} \|e'_{j}(x)y'(y)e_{i}\|_{k}\\ &= \|e'_{j}\|_{S(k)}|y'(y)|\|e_{i}\|_{k}=\frac{c_{i,k}}{a_{j,S(k)}}|y'(y)|.
\end{split}
\end{equation*}
We also have
\begin{equation*}
\begin{split}
\|R\|_{q,\phi(q)}&=\sup_{j\in\nc}\frac{\|y\|_q}{a_{j,\phi(q)}}  \qquad \text{and} \qquad \|Q\|_{q,\phi(q)}=\sup_{i\in\nc}\|y'\|^*_{\phi(q)}c_{i,q}.
\end{split}
\end{equation*}
Applying Theorem \ref{sec:thm2} to the operators $T$, $R$, and $Q$ defined above yields the desired result.

Next suppose that there exists a nondecreasing function $S:\mathbb{N}\to\mathbb{N}$ such that for every nondecreasing function $\phi:\mathbb{N}\to\mathbb{N}$ the condition \ref{P5} holds for every $y\in G$, $y'\in G'$, $i,j\in \nc$. We aim to show that the triple $(\lambda(A), G, \lambda^{\infty}(C))$ has tame factorization property. Let $T\in L(\lambda(A),\lambda^{\infty}(C))$  be an arbitrary operator with $T=RQ$ for some $R\in L(G,\lambda^{\infty}(C))$ and $Q\in L(\lambda(A),G)$.

Let $\phi:\nc \to \nc$ be an arbitrary function. By \ref{P5} there exists a $k_{0}\in \nc$ corresponding to $\phi$.
Fix $k\ge k_0$. Let $y_j=Qe_j$ for every $j\in \mathbb{N}$. By  Lemma~\ref{LE1} (1), it follows that
\begin{equation}\label{E1}
\|Q\|_{q,\phi(q)}= \sup_{j\in \nc} \Biggl\lbrace \frac{\|y_{j}\|_{q}}{a_{j,\phi(q)}}\Biggr\rbrace
\end{equation}
We define $y_i'=e_i'\circ R$ for all $i\in \nc$. Then  we can write
$$Ry=((Ry)_{i})_{i\in \mathbb{N}}= (e'_{i}\circ R (y) )_{i\in \nc}=(y'_i(y))_{i\in \mathbb{N}}$$
and then it follows that
$$\|R\|_{q,\phi(q)}=\sup_{\|y\|_{\phi(q)}\leq 1}\|Ry\|_{q} = \sup_{\|y\|_{\phi(q)}\leq 1} \sup_{i\in \mathbb{N}}|y'_i(y)|c_{i,q}=\sup_{i\in\nc}\|y'_i\|^*_{\phi(q)}c_{i,q}.$$ Then, for every $j\in \nc$, we can write 
$$Te_j=RQe_j=Ry_j=(e_i'(Ry_j))_{i\in\nc}=(y_i'(y_j))_{i\in\nc}=\sum^{\infty}_{i=1} y_i'(y_j)e_{i}$$ 
and for every $x=(x_{j})_{j\in \nc}\in \lambda(A)$, we also have
$$Tx=\sum^{\infty}_{j=1}x_{j}Te_{j} =\sum^{\infty}_{j=1} x_{j}\sum^{\infty}_{i=1} y'_{i}(y_j)e_{i}=\sum^{\infty}_{i=1} \left(\sum^{\infty}_{j=1}x_{j}y'_{i}(y_{j}) \right)e_{i}.$$
For every $x=(x_j)_{j\in\nc}\in \lambda(A)$ with $\|x\|_{S(k)}\leq 1$,
\begin{align*}
\|Tx\|_k&=\sup_{i\in\nc}\Biggl|\sum_{j=1}^\infty x_jy'_i(y_j)\Biggr|c_{i,k} \leq \sup_{i\in\nc}\sum_{j=1}^\infty |x_j|a_{j,S(k)}\frac{c_{i,k}}{a_{j,S(k)}}|y'_i(y_j)|
\\
&\leq C_k\sup_{i\in\nc}\Bigl\{\max_{1\leq q\leq q_k}  \|y'_i\|^*_{\phi(q)}c_{i,q} \Bigr\} \cdot \sup_{j\in\nc}\Biggl\{\max_{1\leq q\leq q_k}\Biggl\lbrace\frac{\|y_j\|_q}{a_{j,\phi(q)}}\Biggr\rbrace\Biggr\rbrace \cdot
\sum_{j=1}^\infty|x_j|a_{j,S(k)}\\
&\leq C_k\max_{1\leq q\leq q_k}\bigl\lbrace \|Q\|_{q,\phi(q)} \bigl\rbrace \max_{1\leq q\leq q_k}\bigl\lbrace \|R\|_{q,\phi(q)}\bigl\rbrace <\infty.
\end{align*}
Therefore, we have
$$\|T\|_{k,S(k)}\leq C_k \max_{1\leq q\leq q_{k}} \bigl\lbrace \|Q\|_{q,\phi(q)}\bigl\rbrace\max_{1\leq q\leq q_{k}}\bigl\lbrace\|R\|_{q,\phi(q)}\bigl\rbrace.$$
Theorem \ref{sec:thm2} says that 
$(\lambda(A),G,\lambda^\infty(C))$ has tame factorization property.
\end{proof}

The following theorem is an immediate consequence of the above results.

\begin{theorem} \label{sec:factor1}
For K\"othe spaces $\lambda^{m_1} (A)$, $\lambda^{m_2}(B)$ and $\lambda^{m_3} (C)$, where $\lambda^{m_1},\lambda^{m_2}\in\{\lambda,\lambda^\infty_0\}$, $\lambda^{m_3}\in\{\lambda,\lambda^\infty\}$, $(\lambda^{m_1} (A), \lambda^{m_2}(B),\lambda^{m_3} (C))\in \mathfrak{TF}$ if and only if there exists a nondecreasing function $S:\mathbb{N}\to\mathbb{N}$ such that for every nondecreasing function $\phi:\mathbb{N}\to\mathbb{N}$ we have  
\begin{gather}\label{sec:kothe}
\begin{split}
\exists k_0\in\mathbb{N} &\quad \forall k\ge k_0,\quad \exists q_k\in\mathbb{N}, C_{k}>0 ~\\	&
\frac{c_{j,k}}{a_{i,S(k)}}\leq C_k\max_{1\leq q\leq q_k} \frac{b_{\nu,q}}{a_{i,\phi (q)}} \max_{1\leq q\leq q_k} \frac{c_{j,q}}{b_{\nu,\phi (q)}}\end{split}\end{gather}
for every $i,j,\nu \in \nc$.
\end{theorem}

\section{Factorization over a Tensor Product}

We now consider triples $(E,G,F)$ of Fr\'echet spaces in which $G$ is the complete projective tensor product of K\"othe spaces $\lambda(A)$ and $\lambda(B)$. $G$ is naturally isomorphic to the K\"othe space $\lambda(C)$, where $C = \{c_{\nu,p}:p \in \mathbb{N}, \nu=(i,j)\in \nc^{2}\}$ is given by 
\begin{equation}\label{TM1}
c_{\nu,p} = a_{i,p} b_{j,p}
\end{equation}
for every $p\in \nc, \nu=(i,j)\in \nc^{2}$.

\begin{theorem}\label{sec:tensor}
	Let $(E,\lambda(A))$ and $(\lambda(B),F)$ be tame pairs. If $G$ is isomorphic to $\lambda(A)\hat{\otimes}_\pi \lambda(B)$ then $(E,G,F)$ satisfies Equation \ref{sec:wfac1}. 
\end{theorem}

\begin{proof}
Let us assume that $G=\lambda(A)\hat{\otimes}_\pi\lambda(B)$ is isomorphic to a K\"othe space $\lambda(C)$ given by the Köthe matrix in Equation \ref{TM1}. 
	
By Theorem 2.3 in \cite{pisek2}, there exists an increasing function $S_{1}:\nc\to \nc$, tending to infinity, such that for any other increasing
function $\phi:\nc \to \nc$ tending to infinity we have:
\begin{equation*}
\begin{split}
\exists k_1\quad \forall k\geq k_1\quad & \exists t_k\in \nc, \, A_k>0 \quad \forall \nu\in \nc, z\in F :\\ &
 \frac{\|z\|_k}{b_{\nu, S_1(k)}}\leq A_k\max_{1\leq q\leq t_k}\Bigl\{\frac{\|z\|_q}{b_{\nu,\phi(q)}}\Bigr\}.
\end{split}
\end{equation*}
	By Theorem 2.2 in \cite{pisek2}, there exists an increasing function $S_{2}:\nc\to \nc$, tending to infinity, such that for any other increasing
function $\phi:\nc \to \nc$ tending to infinity we have:
\begin{equation*}
\begin{split}
\exists k_2\quad \forall n\geq k_2\quad &\exists p_n\in \nc, \, B_n>0 \quad \forall \mu\in \nc, u'\in E': \\ & \|u\|^*_{S_2(n)}a_{\mu,n}\leq B_n\max_{1\leq q\leq p_n}\Bigl\{\|u'\|^*_{\phi(q)}a_{\mu,q}\Bigr\}.
\end{split}
\end{equation*}
	Let $\phi:\nc \to \nc$ be an increasing function and let $k_0=\max\{k_1,k_2,\phi(1)\}$.  For any $k\geq k_0$, without loss of generality, we can take $t_k\geq S_1(k)\geq k_0$. Then the second inequality is true for $n=S_1(k)$, and similarly we can take $p_k(=p_{S_1(k)})\geq k_2$. Moreover, for $q_k=\max\{t_k,p_k\}$, we have:
	\begin{gather*}
		\frac{1}{a_{\mu,S_1(k)}}\leq \frac{1}{a_{\mu,\phi(1)}}\leq \max_{1\leq q\leq q_k}\Bigl\{\frac{1}{a_{\mu,\phi(q)}}\Bigr\}\quad \textrm{and}\quad 
		b_{\nu,S_1(k)}\leq b_{\nu,q_k}\leq \max_{1 \leq q\leq q_k}\Bigl\{b_{\nu,q}\Bigr\}.
	\end{gather*}
	Combining these and letting $S=S_2\circ S_1$, given $\phi$, when $k\geq k_0$, $C_k=A_kB_k$ we have:
	\begin{align*}
		\|u'\|^*_{S(k)}\|z\|_k&=\frac{\|z\|_k}{b_{\nu,S_1(k)}}\frac{b_{\nu, S_1(k)}}{a_{\mu,S_1(k)}}\|u\|^*_{S_2(S_1(k))}a_{\mu,S_1(k)}\\
		&\leq C_k\max_{1\leq q\leq q_k}\Bigl\{\frac{\|z\|_q}{a_{\mu,\phi(q)}b_{\nu, \phi(q)}}\Bigr\}\max_{1\leq q\leq q_k}\Bigl\{\|u'\|^*_{\phi(q)}a_{\mu ,q}b_{\nu,q}\Bigr\},
	\end{align*} 
	which is the desired inequality.
\end{proof}

\begin{theorem}\label{TFT}
	Let $G$ be isomorphic to the space $\lambda(B)\otimes_\pi\lambda(C)$. If the pairs $(\lambda(A),\lambda(B))$ and $(\lambda(C),\lambda(D))$ are tame then $(\lambda(A),G,\lambda(D))$ has tame factorization property. 
\end{theorem}

\begin{proof}
	Let $(\lambda(A),\lambda(B))$ and $(\lambda(C),\lambda(D))$ be tame pairs. Then by Theorem \ref{sec:tensor} Equation \ref{sec:wfac1} is satisfied. Then by Theorem \ref{sec:factor1} $(\lambda(A),G,\lambda(D))$ has tame factorization property.
\end{proof}

\section{Quasi-diagonal Operators}

A quasi-diagonal operator $T:\lambda(A)\to \lambda(B)$ is an operator satisfying 
$$Te_{n}=t_{n}e_{\sigma(n)}$$ 
for every $n\in \nc$, where $(t_{n})_{n\in \nc}$ is a sequence of scalars and $\sigma:\nc\to \nc$ is an injective map. 

Dragilev \cite{drag} and Nurlu \cite{nurlu}\cite{nurlu2} proved that the existence of a continuous unbounded linear operator between nuclear Köthe spaces implies the existence of a continuous unbounded quasi-diagonal operator between them. Djakov and Ramanujan \cite{djakov2} later sharpened these results by removing the nuclearity assumption. Building on these results, Bahadır showed in \cite[Theorem 3.1]{can} that the pair 
$(\lambda(A),\lambda(B))$ is tame if and only if every continuous quasi-diagonal operator from $\lambda(A)$ to $\lambda(B)$ is tame. Moreover, Terzioğlu et al. in \cite{terzi1} proved that if there exists an unbounded continuous linear operator $T:\lambda(A)\to\lambda(C)$ that factors through $\lambda(B)$, then there also exists an unbounded continuous quasi-diagonal operator $D:\lambda(A)\to\lambda(C)$ which factors through $\lambda(B)$ as a product of two continuous quasi-diagonal operators. We now want to expand this result to tame factorization property. 

We denote by $\mathcal{D}(A,B)$ the set of all continuous quasi-diagonal operators from $\lambda(A)$ to $\lambda(B)$, and by $\mathcal{D}_\sigma(A,B)$ the space of all continuous quasidiagonal operators from $\lambda(A)$ to $\lambda(B)$ determined by the  map $\sigma$. Clearly, $\mathcal{D}_\sigma(A,B)$ is a subspace of $L(\lambda(A),\lambda(B))$, whereas $\mathcal{D}(A,B)$ is not, and we have $$\mathcal{D}_\sigma(A,B)\subseteq \mathcal{D}(A,B)\subseteq L(\lambda(A),\lambda(B)).$$ 

	Given any $\phi:\nc\to\nc$, and an injective map $\sigma$, we define:
\begin{gather*}
	\mathcal{D}_\sigma^\phi(A,B):=\mathcal{D}_\sigma(A,B)\cap L_\phi(\lambda(A),\lambda(B)). 
\end{gather*} 
Since we are working with K\"othe spaces, $\mathcal{D}_\sigma^\phi(A,B)$ is a closed subspace of $L_\phi(\lambda(A),\lambda(B))$. 

\begin{pro}\label{PD1}
	Let $\sigma,\tau:\nc\to\nc$ be injective maps and $S:\nc\to \nc$ be a nondecreasing map. If $RQ$ is $S$-tame for each $Q\in \mathcal{D}_\sigma(A,B)$, $R\in \mathcal{D}_\tau(B,C)$, then for each $\phi:\nc\to\nc$  we have 
    \begin{equation}\label{sec:qdineq1}
\begin{split}
\exists k_0\in\mathbb{N} \quad &\forall k\ge k_0,\quad \exists n_k\in\mathbb{N} ~\\	&\frac{c_{\tau\circ\sigma(i), k}}{a_{i,S(k)}}\leq n_k\max_{1\leq q\leq n_k}\Biggl\lbrace \frac{b_{\sigma(i) ,q}}{a_{i,\phi (q)}}\Biggr\rbrace \max_{1\leq q\leq n_k}\Biggl\lbrace \frac{c_{\tau\circ\sigma(i),q}}{b_{\sigma(i),\phi (q)}}\Biggr\rbrace
\end{split}
\end{equation}
for every $i\in \nc$.
\end{pro}

We note that here $k_0$ and $n_k$ depend not just on $\phi$, but also on the choice of $\sigma$ and $\tau$. 

\begin{proof}
For the proof, we will follow the methods used in \cite[Proposition 3]{terzi1}.
Let $\sigma,\tau$ be fixed injective maps, and take $Q\in\mathcal{D}_\sigma(A,B)$, $R\in\mathcal{D}_\tau(B,C)$ so that $RQ$ assumed to be $S$-tame for some nondecreasing map $S:\nc\to\nc$. Let $\phi:\nc\to\nc$ be an arbitrary function.
	
Since $Q\in\mathcal{D}_\sigma(A,B)$ and $R\in\mathcal{D}_\tau(B,C)$, there exist sequences $(q_i)_{i\in \mathbb{N}}$ and $(r_{i})_{i\in \nc}$ satisfying  $$Qe_i:=q_ie_{\sigma(i)}\qquad \text{and} \qquad Re_i:=r_ie_{\tau(i)}$$
for all $i\in \nc$.

 For each $i,j\in\nc$, we define operators $Q_i:\lambda(A)\to\lambda(B)$ and $R_j:\lambda(B)\to\lambda(C)$ by
\[
Q_ix:=x_iQe_i, \qquad R_jy:=y_jRe_j,
\]
for every $x=(x_i)_{i\in\nc}\in\lambda(A)$ and $y=(y_i)_{i\in\nc}\in\lambda(B)$. Clearly $Q_i\in\mathcal{D}_\sigma(A,B)$ and $R_j\in\mathcal{D}_\tau(B,C)$ for all $i,j\in\nc$.

We further define $T_i:=R_{\sigma(i)}Q_i$ for each $i\in\nc$, so that
\[
T_ix = x_i\,q_i\,r_{\sigma(i)}\,e_{\tau(\sigma(i))}
\]
for every $x=(x_i)_{i\in\nc}\in\lambda(A)$ and each $T_i$ is $S$-tame. It follows that
\begin{equation*} 
\begin{split}
& \|Q_i\|_{q,\phi(q)}=|q_i|\frac{b_{\sigma(i),q}}{a_{i,\phi(q)}},\\ &\|R_j\|_{q,\phi(q)}=|r_j|\frac{c_{\tau(j),q}}{b_{j,\phi(q)}},\\ & \|T_i\|_{k,S(k)}=|q_i||r_{\sigma(i)}|\frac{c_{\tau\circ\sigma(i),k}}{a_{i,S(k)}}.
\end{split}
\end{equation*}

We first note that if $Q\notin\mathcal{D}_\sigma^\phi(A,B)$ or $R\notin\mathcal{D}_\tau^\phi(B,C)$, then the same holds for all $Q_i$, respectively all $R_j$. This means that for a given $q\in\nc$, either $\|Q_i\|_{q,\phi(q)}=\infty$ for all $i\in\nc$, or $\|R_j\|_{q,\phi(q)}=\infty$ for all $j\in\nc$. In either case, the right-hand side of Inequality \ref{sec:qdineq1} equals $\infty$, so the inequality holds trivially.
	
So we can assume that $Q\in \mathcal{D}_\sigma^\phi(A,B)$ and $R\in \mathcal{D}_\tau^\phi(A,B)$. The map $\theta$ defined by 
\[\theta: \mathcal{D}_\sigma^\phi(A,B)\times \mathcal{D}_\tau^\phi(B,C)\to \bigcup_{n\in\nc}L_{S,n}(\lambda(A),\lambda(B),\quad \theta(Q,R)=RQ\]
is sequentially separately closed. We apply Lemma \ref{BFL} \cite[Lemma~2.1]{terzi2} to get that there is some $k_{0}\in \nc$ such that 
\[ \theta: \mathcal{D}_\sigma^\phi(A,B)\times \mathcal{D}_\tau^\phi(B,C)\to L_{S,k_{0}}(\lambda(A),\lambda(B)\]
is continuous. Since \[T_i=R_{\sigma(i)}Q_i=\theta(Q_i\times R_{\sigma(i)}),\quad Q_i\times R_{\sigma(i)}\in \mathcal{D}_\sigma^\phi(A,B)\times \mathcal{D}_\tau^\phi(B,C)\] for each $i\in\nc$, the continuity of $\theta: \mathcal{D}_\sigma^\phi(A,B)\times \mathcal{D}_\tau^\phi(B,C)\to L_{S,k_{0}}(\lambda(A),\lambda(B)$ yields that for every $ k\geq k_0$ there exists $n_k\in \nc$ such that 
 \begin{align*}
	\|T_{i}\|_{k,S(k)} &\leq n_k\max_{1\leq q\leq n_k}\|R_{\sigma(i)}\|_{q,\phi(q)}\max_{1\leq q\leq n_k}\|Q_{i}\|_{q,\phi(q)}
\end{align*} 
 and 
\[\frac{c_{\tau\circ\sigma(i), k}}{a_{i,S(k)}}\leq n_k\max_{1\leq q\leq n_k}\Biggl\lbrace \frac{b_{\sigma(i) ,q}}{a_{i,\phi (q)}}\Biggr\rbrace \max_{1\leq q\leq n_k}\Biggl\lbrace \frac{c_{\tau\circ\sigma(i),q}}{b_{\sigma(i),\phi (q)}}\Biggr\rbrace\]
for every $i\in \nc$. This gives the desired result. 
\end{proof}

\begin{observation}\label{sec:kothecor}
If $(\lambda(A),\lambda(B),\lambda(C))\notin\mathfrak{TF}$ then, by Theorem \ref{sec:factor1}, for any nondecreasing function $S:\mathbb{N}\to\mathbb{N}$ there exists a nondecreasing function $\phi:\mathbb{N}\to\mathbb{N}$ such that
\begin{equation}\label{sec:kothe}
\begin{split}
\forall k\in\mathbb{N} \quad &\exists k_{0}\in\nc,\ k_{0}\geq k \quad \forall n\in\mathbb{N} \\
&\frac{c_{j_n,k_0}}{a_{i_n,S(k_0)}} > n \max_{1\le q\le n}\frac{b_{\nu_n, q}}{a_{i_n,\phi(q)}}\max_{1\le q\le n}\frac{c_{j_n,q}}{b_{\nu_n,\phi(q)}}
\end{split}
\end{equation}
holds for some $i_n(k_0),j_n(k_0),\nu_n(k_0)\in\mathbb{N}$.
\end{observation}

\begin{lemma}
Let $S,\phi:\mathbb{N}\to\mathbb{N}$ be nondecreasing functions. Suppose that 
\begin{equation}
\begin{split}
\forall k\in\mathbb{N} \quad &\exists k_{0}\in \nc, k_{0}\geq k \quad \forall n\in\mathbb{N} ~\\	& \frac{c_{j_n k_0}}{a_{i_nS(k_0)}} > n \max_{1 \le q \le n} \frac{b_{\nu_n q}}{a_{i_n\phi(q)}} \max_{1 \le q \le n} \frac{c_{j_n q}}{b_{\nu_n \phi(q)}}
\end{split}
\end{equation}
holds for some $i_{n}(k), j_{n}(k),\nu_{n}(k)\in\mathbb{N}$.
Then, for each fixed $k$, the sequences $(i_n(k))_{n\in\nc}$, $(j_n(k))_{n\in\nc}$, $(\nu_n(k))_{n\in\nc}$ diverge to $\infty$ as $n\to\infty$.
\end{lemma}

\begin{proof}
This follows essentially from Lemma 4 in \cite{terzi1}, taking $r=S(k_0)$ and $q=k_0$.
\end{proof}

\begin{pro}\label{sec:nontame}
	If $(\lambda(A),\lambda(B),\lambda(C))\notin \mathfrak{TF}$ then for any given $S:\nc\to \nc$ there are bijections $\sigma$ and $\tau$ and operators $Q\in \mathcal{D}_\sigma(A,B)$, $R\in \mathcal{D}_\tau(A,B)$ such that the operator $T=RQ$ is not $S$-tame. 
\end{pro}

\begin{proof} Let us assume that  $(\lambda(A),\lambda(B),\lambda(C))\notin\mathfrak{TF}$. Fix $S:\nc\to\nc$. Observation \ref{sec:kothecor} guarantees the existence of a nondecreasing function $\phi:\nc\to\nc$, together with, for each $n\in\nc$, indices $i_n,j_n,\nu_n\in\nc$, for which \eqref{sec:kothe} holds. Following the technique of \cite[Proposition 5]{terzi1}, we construct, for each $k\in \nc$, a sequence $U_k=\lbrace n_m(k)\rbrace$ such that each coordinate of $(i_{n_m}(k),j_{n_m}(k),\nu_{n_m}(k))$ takes different values for different $m$. For each $k\in \nc$, we represent infinite set $U_{k}$ as a disjoint union of infinite subsets
$$U_k=\bigcap_{\mu=0}^\infty U_{k,\mu}.$$

Now we construct a new sequence of infinite disjoint sets inductively. First form $V_{k_0}$ by taking one element $t_\mu(k_0)$ from each $U_{k_0,\mu}$, $\mu \in \nc$. We construct $V_{k+1}$ by taking one element from each $U_{k+1,\mu}$, different from $t_\mu(s)$ for $k_0\leq s\leq k$. Then each of the sets 
$$I_0:= \nc \setminus \bigcup_{k=k_0}^\infty I_{V_k}, \quad J_0:= \nc \setminus \bigcup_{k=k_0}^\infty J_{V_k}, \quad N_0:= \nc \setminus \bigcup_{k=k_0}^\infty N_{V_k}$$
is infinite. 
	
Let $\eta:J_0\to N_0$ and $\zeta:N_0\to I_0$ be arbitrary bijections. We define the maps $\sigma$ and $\tau$ by:
	\begin{equation*}
		\sigma(i) := \left(
		\begin{array}{cc}
			\eta(i) & \textrm{if }i\in J_0 \\
			j_{t_\mu(k)} & \textrm{if }i=i_{t_\mu(k)}\in J_{V_k},\,k\geq k_0, 
		\end{array} \right)
	\end{equation*}
	
	\begin{equation*}
		\tau(j) := \left(
		\begin{array}{cc}
			\zeta(j) & \textrm{if }j\in N_0 \\
			\nu_{t_\mu(k)} & \textrm{if }j=j_{t_\mu(k)}\in N_{V_k},\,k\geq k_0.
		\end{array} \right)
	\end{equation*}
Then for each $k$ we have:
	
	$$\frac{c_{\tau(\sigma(i)),k}}{a_{i,S(k)}}> n \max_{1\leq q\leq n}\frac{b_{\sigma(i),q}}{a_{i,\phi(q)}}\cdot \max_{1\leq q\leq n}\frac{c_{\tau(\sigma(i)),q}}{b_{\sigma(i),\phi(q)}}$$
	for every $i=i_n$ with $n\in V_k$. Proposition \ref{PD1} gives us the desired result.
\end{proof}

\begin{theorem}
	$(\lambda (A),\lambda(B),\lambda(C))\in \mathfrak{TF}$ with associated function $S:\nc\to\nc$ if and only for each $Q\in \mathcal{D}(A,B)$ and $R\in \mathcal{D}(B,C)$, the operator $T=RQ$ is $S$-tame. 
\end{theorem}

\begin{proof}
Clearly, if the triple has the tame factorization property, then there exists $S:\nc\to\nc$ such that every $T=RQ$ with $Q\in\mathcal{D}(A,B)$, $R\in\mathcal{D}(B,C)$ is $S$-tame. Conversely, if $(\lambda(A),\lambda(B),\lambda(C))\notin\mathfrak{TF}$, then by Proposition \ref{sec:nontame} there exist bijections $\sigma,\tau$ and operators $Q\in\mathcal{D}_\sigma(A,B)$, $R\in\mathcal{D}_\tau(B,C)$ such that $T=RQ$ is not $S$-tame. This completes the proof.
\end{proof}

\section{Applications to Power Series Spaces}

In this section, we specialize the results established in the preceding sections and \cite{can} to triples of power series spaces. Since power series spaces constitute the most concrete and widely studied class of K\"othe spaces, the general characterizations of $\mathfrak{TF}$ and $\mathfrak{BF}$ established above lead directly to several explicit results for these triples. We collect these consequences below, all of which follow from short applications of preceding sections.

Let us consider a triple $(E,G,F)$, where $E=\Lambda_{1}(\alpha)$ and $F=\Lambda_t(\beta)$ are power series spaces, $t\in\{1,\infty\}$. Since the pair $(\Lambda_1(\alpha),\Lambda_t(\beta))$ is tame for all sequences $\alpha,\beta$, \cite{vogt3,Z2} it follows that $(\Lambda_1(\alpha),G,\Lambda_t(\beta))\in\mathfrak{TF}$ for every Fr\'echet space $G$ and every $t\in\{1,\infty\}$.

A sequence $\beta$ is called \emph{stable} if $$\displaystyle \sup_{j\in\mathbb{N}} \frac{\beta_{j+1}}{\beta_j}<\infty,$$ and \emph{non-stable} otherwise. The tameness of pairs of power series spaces in each of these cases was completely determined by Bahad\i r. We directly adopt these results, summarized in Table~1 of \cite{can}, in what follows.

\begin{pro}\label{prop:powerseries}
Let $\alpha,\beta$ be increasing sequences, with $\alpha$ non-stable, and let $t\in\{1,\infty\}$.
\begin{enumerate}
\item[(i)] If $\beta$ is also non-stable, then $(\Lambda_\infty(\alpha),G,\Lambda_t(\beta))\in\mathfrak{TF}$ for every Fr\'echet space $G$.
\item[(ii)] If $\beta$ is stable, and $G=\Lambda_1(\gamma)$ or $G=\Lambda_r(\gamma)\hat{\otimes}_\pi\Lambda_1(\delta)$ for some $r\in\{1,\infty\}$, where $\gamma$ is non-stable and $\delta$ is an arbitrary sequence, then $(\Lambda_\infty(\alpha),G,\Lambda_t(\beta))\in\mathfrak{TF}$.
\end{enumerate}
\end{pro}

\begin{proof}
Let $\alpha$ be non-stable.
\begin{itemize}
\item[(i)] When $\beta$ is non-stable, $(\Lambda_\infty(\alpha),\Lambda_t(\beta))$ is tame by Theorem 4.3 of \cite{can}, so the result is trivial by Remark~\ref{Rem1}(3).
\item[(ii)] Suppose $\beta$ is stable. If $G=\Lambda_1(\gamma)$, then $(\Lambda_\infty(\alpha),\Lambda_1(\gamma))$ is tame since $\alpha$ and $\gamma$ are both non-stable by Theorem 4.5 and Remark 4.6 of \cite{can}, and $(\Lambda_1(\gamma),\Lambda_t(\beta))$ is tame for every $t\in\{1,\infty\}$ regardless of the stability of $\gamma$ and $\beta$; hence $(\Lambda_\infty(\alpha),\Lambda_1(\gamma),\Lambda_t(\beta))\in\mathfrak{TF}$ by Remark~\ref{Rem1}(1). If $G=\Lambda_r(\gamma)\hat\otimes_\pi\Lambda_1(\delta)$, the same reasoning shows that $(\Lambda_\infty(\alpha),\Lambda_r(\gamma))$ and $(\Lambda_1(\delta),\Lambda_t(\beta))$ are both tame, so the conclusion follows from Theorem~\ref{sec:tensor}.
\end{itemize}
\end{proof}

\begin{pro}\label{thm:tensorpowerseries}
Let $\alpha$ and $\beta$ be stable sequences.  If $\lambda(B)$ has $(LB_\infty)$ and $\lambda(C)$ has $(LB^\infty)$, 
$(\Lambda_\infty(\alpha),\lambda(B)\hat{\otimes}_\pi\lambda(C),\Lambda_1(\beta))\in \mathfrak{TF}$.
\end{pro}

\begin{proof}
	Since $\lambda(B)$ has $(LB_\infty)$ and $\lambda(C)$ has $(LB^\infty)$, then  the pairs $(\Lambda_\infty(\alpha),\lambda(B))$ and $(\lambda(C), \Lambda_\infty(\beta))$ is bounded by 2.1 Satz and 3.2 Satz of \cite{vogt2}. Hence, Theorem \ref{TFT} gives us that   $(\Lambda_\infty(\alpha),\lambda(B)\hat{\otimes}_\pi\lambda(C),\Lambda_\infty(\beta))\in \mathfrak{TF}$.
\end{proof}

\begin{rem}
When $\alpha$ and $\beta$ are stable, the pair $(\Lambda_\infty(\beta),\Lambda_1(\alpha))$ need not be tame by Theorem 4.8 of \cite{can}. Nevertheless, Proposition~\ref{thm:tensorpowerseries} shows that operators factoring through a suitably chosen tensor product $G=\lambda(B)\hat{\otimes}_\pi\lambda(C)$ are tame, once again illustrating that $\mathfrak{TF}$ is a genuinely weaker and more flexible condition than pairwise tameness.
\end{rem}

As a further consequence of the tame factorization property, we obtain the following result on the existence of bases, due to Bahad\i r \cite[Theorem 5.1]{can}.

\begin{theorem}\label{thm:basis}
Suppose $(\Lambda_\infty(\alpha),G,\Lambda_\infty(\beta))$ has the tame factorization property. Then the range of any operator $T\in L^G(\Lambda_\infty(\alpha),\Lambda_\infty(\beta))$ has a basis. If, in addition, $\Lambda_\infty(\beta)$ is nuclear, then the range of $T$ has an absolute basis; and if the range of $T$ is also closed, then it is isomorphic to a closed subspace of $s$, the space of rapidly decreasing sequences $\Lambda_\infty\big((\ln n)_{n\in\mathbb{N}}\big)$.
\end{theorem}

\begin{rem}
In particular, Theorem~\ref{thm:basis} applies whenever $G$ is one of the spaces appearing in Proposition~\ref{prop:powerseries}(ii).
\end{rem}

	
	


\section*{Acknowledgments}
This work has been supported by Fatih Sultan Mehmet Vakıf University Research Projects Coordination Unit under grant number 26FSMBC1FB008.

\end{document}